\documentclass[pdflatex,sn-mathphys-num]{sn-jnl}% Math and Physical Sciences Numbered Reference Style
\usepackage{graphicx}%
\usepackage{multirow}%
\usepackage{amsmath,amssymb,amsfonts}%
\usepackage{amsthm}%
\usepackage{mathrsfs}%
\usepackage[title]{appendix}%
\usepackage{xcolor}%
\usepackage{textcomp}%
\usepackage{manyfoot}%
\usepackage{booktabs}%
\usepackage{algorithm}%
\usepackage{algorithmicx}%
\usepackage{algpseudocode}%
\usepackage{listings}%
\usepackage{changes}
\theoremstyle{thmstyleone}%
\newtheorem{theorem}{Theorem}%  meant for continuous numbers
\newtheorem{lemma}[theorem]{Lemma}
\newtheorem{corollary}[theorem]{Corollary}
\theoremstyle{thmstyletwo}%
\newtheorem{remark}{Remark}%

\theoremstyle{thmstylethree}%
\newtheorem{definition}{Definition}%
\newcommand{\diam}{\textrm{diam}}
\newcommand{\mis}{\textnormal{MIS}}

\begin{document}

\title[Article Title]{Resonance graphs that are daisy cubes: from  hypercubes to independent sets via resonant sets  }

%%=============================================================%%
%% GivenName	-> \fnm{Joergen W.}
%% Particle	-> \spfx{van der} -> surname prefix
%% FamilyName	-> \sur{Ploeg}
%% Suffix	-> \sfx{IV}
%% \author*[1,2]{\fnm{Joergen W.} \spfx{van der} \sur{Ploeg} 
%%  \sfx{IV}}\email{iauthor@gmail.com}
%%=============================================================%%

\author[a,b]{\fnm{Simon} \sur{Brezovnik}}\email{simon.brezovnik@fs.uni-lj.si}

\author[c]{\fnm{Zhongyuan} \sur{Che}}\email{zxc10@psu.edu}
%\equalcont{These authors contributed equally to this work.}

\author*[d,b]{\fnm{Niko } \sur{Tratnik}}\email{niko.tratnik@um.si}
\author[e,d]{\fnm{Petra} \sur{\v Zigert Pleter\v sek}}\email{petra.zigert@um.si}

%\equalcont{These authors contributed equally to this work.}
%\affil*[1]{\orgdiv{Department}, \orgname{Organization}, \orgaddress{\street{Street}, \city{City}, \postcode{100190}, \state{State}, \country{Country}}}
 
\affil[a]{\orgdiv{Faculty of Mechanical Engineering}, \orgname{University of Ljubljana}, \orgaddress{\city{Ljubljana}, \country{Slovenia}}}

\affil[b]{ \orgname{Institute of Mathematics, Physics and Mechanics}, \orgaddress{\city{Ljubljana}, \country{Slovenia}}}

\affil[c]{\orgdiv{Department of Mathematics}, \orgname{Penn State University Beaver Campus}, \orgaddress{\city{Monaca},  \state{PA}, \country{USA}}}

\affil*[d]{\orgdiv{Faculty of Natural Sciences and Mathematics}, \orgname{University of Maribor}, \orgaddress{\city{Maribor}, \country{Slovenia}}}

\affil[e]{\orgdiv{Faculty of Chemistry and Chemical Engineering}, \orgname{University of Maribor}, \orgaddress{\city{Maribor}, \country{Slovenia}}}

%%==================================%%
%% Sample for unstructured abstract %%
%%==================================%%

\abstract{Let $G$ be a plane elementary bipartite graph whose infinite face is forcing. 
We provide a bijection between the set of maximal hypercubes of its resonance graph and the
set of maximal resonant sets of $G$, which generalizes a main result  in 
[MATCH Commun. Math. Comput. Chem. 68 (2012) 65–77],
where $G$ was only considered as an elementary benzenoid graph without nice coronenes.
For a special case when $G$ is a peripherally 2-colorable graph, it follows that  
there is a bijection between the set of maximal hypercubes of its resonance graph 
and the set of maximal independent sets of a tree that is the inner dual of $G$. 
We then show that  the resonance graph of a plane bipartite graph $G$ is a daisy cube
if and only if it is the simplex graph of the complement of a forest.
Finally, we characterize trees with at most 5 maximal independent sets 
to determine daisy cubes that are simplex graphs 
of the complements of trees and having at most five maximal vertices.}


%%\pacs[JEL Classification]{D8, H51}


\maketitle

\section{Introduction}\label{sec1}

Resonance graphs, also termed as $Z$-transformation graphs, depict the relationships among perfect matchings, 
also known as Kekulé structures in chemistry. Originating separately from the works of chemists  
\cite{grundler-82,el-basil-93/1,el-basil-93/2} 
and mathematicians \cite{ZGC88}, these graphs were initially focused on benzenoid graphs, and
later expanded to include plane bipartite graphs \cite{LZ03, ZLS08, ZZ00, ZZY04}.  
This extension broadened the applicability of the concept. 
Research on resonance graphs of benzenoid graphs involving hypercubes is 
documented in several sources such as \cite{SKG06, SKVZ09, TZ12}, 
while some structural aspects of resonance graphs of plane bipartite graphs are summarized in an early survey paper \cite{Z06}. 

The concept of daisy cubes was introduced in \cite{KM19} 
as a subfamily of partial cubes which contains Fibonacci cubes $\Gamma_n$ 
and Lucas cubes $\Lambda_n$. 
For a poset $(P,\le )$, a subset $A\subseteq P$ is a \textit{downward closed set} 
if for any $ x \in A $ and $ y\in P$, if $y \le x$, then $y \in A$.
A \textit{daisy cube} is an induced subgraph of hypercube $Q_n$ 
whose vertex set is a downward closed subset of poset $(\mathcal{B}^n, \le)$,
where $\mathcal{B}^n$ is the set of binary strings of length $n$.
Recent advancements, including characterizations and embeddings of daisy cubes, were discussed in \cite{T20, V21}. 
The linkage between resonance graphs and daisy cubes was initially explored in \cite{Z18}, 
revealing that the resonance graph of a kinky benzenoid graph is a daisy cube. 
Moreover, characterizations of  
catacondensed even ring systems and 2-connected outerplane bipartite graphs whose resonance graphs are daisy cubes 
were provided in \cite{br-tr-zi-1} and \cite{br-tr-zi-2}, respectively. 
We further characterized plane bipartite graphs whose resonance graphs are daisy cubes 
in terms of peripherally 2-colorable graphs in \cite{BCTZ25}:
The resonance graph  of a  plane bipartite graph $G$ is a daisy cube if and only if $G$ is a plane weakly elementary 
bipartite graph whose  elementary components with more than two vertices are peripherally 2-colorable.

Let $G$ be a graph, The  \textit{complement} of $G$, denoted by $\bar{G}$, is the graph with the same vertex set as $G$
such that two vertices are adjacent in $\bar{G}$ if and only if they are not adjacnet in $G$.
A \textit{clique} of $G$ is a subgraph of $G$ such that every pair of vertices is adjacent. 
The \textit{simplex graph} $\mathcal{K}(G)$ \cite{BV89} 
is the graph whose vertices are the cliques of $G$ (including the empty set),
and two cliques are adjacent if they differ in exactly one vertex. 
Both Fibonacci cubes and Lucas cubes are simplex graphs: $\Gamma_n=\mathcal{K}(\bar{P}_n)$ \cite{FMZ99, EKM23} and 
$\Lambda_n=\mathcal{K}(\bar{C}_n)$ \cite{MPZ01},
where $P_n$ (respectively, $C_n$) is a path (respectively, a cycle) on $n$ vertices.
Note that $K$ is a clique of $G$ if and only if $\bar{K}$ is an \textit{independent set} of $\overline{G}$.
The \textit{simplex graph} $\mathcal{K}(G)$ can be defined 
equivalently as the graph whose vertices are the independent sets of $\bar{G}$ (including the empty set),
and two independent sets are adjacent if they differ in exactly one vertex.
This idea was mentioned on page 106 in \cite{EKM23},
Codara and D'Antona \cite{CD16} applied the above fact and introduced a generalization of a
Fibonacci cube (respectively, a Lucas cube) by considering the Hasse diagram of 
the independent sets of the $h$-th power of a path (respectively, a cycle) ordered by inclusion.

In this paper, 
we first show that if $G$ is a plane elementary bipartite graph whose infinite face is forcing, 
then there is a bijection between the set of maximal hypercubes of its resonance graph and the
set of maximal resonant sets of $G$, which generalizes a main result on elementary benzenoid graphs
without nice coronenes in \cite{TZ12}. 
Peripherally 2-colorable graphs are a special type of plane elementary bipartite graphs whose infinite face is forcing.
For a peripherally 2-colorable graph $G$, the above result implies a bijection  between the set of maximal hypercubes of 
its resonance graph and the set of maximal independent sets of its inner dual, which is a tree.
We then show that the resonance graph of a plane bipartite graph is a daisy cube
if and only if it is the simplex graph of the complement of a forest. 
In particular, the resonance graph of a plane elementary  bipartite graph is a daisy cube
if and only if it is the simplex graph of the complement of a tree. 
As applications of our main results, some well known properties of {Fibonacci cubes} 
are provided as examples with different proofs.
Finally, we characterize trees with at most 5 maximal independent sets to determine  
daisy cubes with at most five maximal vertices that are simplex graphs of the complements of trees, 
that is, daisy cubes with at most five maximal vertices and isomorphic to the resonance graphs of plane elementary bipartite graphs.

\section{Preliminaries}\label{S:2}

\subsection{Perfect matchings of plane bipartite graphs}

A \textit{perfect matching} of a graph $G$ is a set $M \subseteq E(G)$ of vertex disjoint edges
such that every vertex of  $G$ is incident to an edge from $M$. 
A path (or, a  cycle) of $G$ is called \textit{$M$-alternating} if its edges are  in  and  out of $M$ alternately.  
Moreover,   a path (or, a cycle) of $G$  is called \textit{$(M_1, M_2)$-alternating} 
if its edges are {contained in $M_1$ and $M_2$ alternately}
for distinct perfect matchings $M_1$ and $M_2$ of $G$.
A path $P$ of  $G$ is called a \textit{handle} if all internal vertices (if exist) of $P$ are degree-2 vertices of $G$,
and each end vertex of $P$ has degree at least three in $G$ \cite{CC13}. 
It is clear that any handle of $G$ with more than one edge is $M$-alternating for any perfect matching $M$ of $G$.
A subgraph $H$ of a graph $G$ is said to be \textit{nice} if the induced subgraph $G-V(H)$ has a perfect matching.

Let $G$ be a plane graph. A face $s$ of $G$ is a region enclosed by a set of edges of $G$, 
which forms the \textit{periphery of $s$}, denoted by $\partial s$.
A face $s$ of $G$  is called \textit{finite} if the periphery of $s$ encloses a finite region,  
and \textit{infinite} otherwise. The periphery of the infinite face is called the \textit{periphery of $G$}.
A face $f$ of  $G$ is called \textit{$M$-resonant} if the periphery of $f$ 
is an $M$-alternating cycle for a perfect matching $M$ of $G$, 
and $f$ is called \textit{$(M_1, M_2)$-resonant} if its periphery is an 
$(M_1, M_2)$-alternating cycle for two different perfect matching $M_1$ and $M_2$ of $G$.
The \textit{inner dual}  of $G$, denoted by $G^*$, is a graph whose vertices are finite faces of $G$ 
such that two vertices are adjacent in $G^*$  if two corresponding finite faces of $G$ have an edge in common.
Vertices  on the periphery of $G$ are called \textit{exterior vertices},
and the remaining vertices  are \textit{interior vertices} of $G$. 
A  plane graph whose vertices are all exterior vertices is called an \textit{outerplane graph}.

Let $G$ be a plane bipartite graph with a perfect matching.
Let $S$ be a set of finite faces of $G$. 
Then $S$ is called a \textit{resonant set} of $G$ if all finite faces in $S$ are pairwise vertex disjoint,
and there exists a perfect matching $M$ of $G$ such that each  face in $S$ is $M$-resonant.
Moreover, a resonant set of $G$ is called \textit{maximal} if it is not contained in any other  resonant set of $G$. 
For simplicity, we use $G-S$ to represent  the induced subgraph of $G$ 
obtained by removing all vertices of finite faces from $S$.  It is easy to see that if $S$ is an alternating set of $G$,
then the induced subgraph $G-S$ is either empty or has a perfect matching. 
Furthermore, a resonant set $S$ of $G$ is called a \textit{canonical resonant set} if $G-S$ is empty or has a unique perfect matching.

An edge of a graph with a perfect matching is called \textit{allowed} if it is contained in a perfect matching of the graph, 
and \textit{forbidden} otherwise.
A graph is said to be \textit{elementary} if all its allowed edges form a connected subgraph. 
For a graph with a perfect matching, each component of its subgraph obtained  by removing all forbidden edges
is elementary.
It was shown that a bipartite graph $G$ is elementary if and only if it is connected and each edge of $G$ is allowed \cite{LP09}.
In addition, a connected plane bipartite graph with more than two vertices  is elementary 
if and only if the periphery of each face (including the infinite face) 
is an alternating cycle with respect to some perfect matching of the graph \cite{ZZ00}.
A  face $s$ of a plane graph $G$ is said to be a \textit{forcing face} of $G$ if the induced subgraph of $G$ 
obtained by deleting all vertices of $s$  is either empty or has a unique perfect matching.
By \cite{CC13}, any 2-connected plane bipartite graph with a forcing face is elementary.
The concept of a plane weakly elementary bipartite graph was first introduced in \cite{ZZ00} for connected graphs.
It is more practical to extend the concept to include non-connected graphs too (for example, see \cite{ZZY04}).
A  plane bipartite graph (not necessarily connected) with a perfect matching is called \textit{weakly elementary} 
if deleting all forbidden edges does not produce any new finite face. 
By definition, any plane elementary bipartite graph is also weakly elementary.
Note that all allowed edges of a plane weakly elementary bipartite graph $G$ form \textit{elementary components} of $G$,
and each elementary component of $G$ different from $K_2$ is called \textit{nontrivial}.

A \textit{benzenoid system} is a 2-connected subgraph of the hexagonal lattice such that every
finite face is a hexagon. It is well known that every benzenoid system with a perfect matching is weakly elementary  \cite{Z06}.

We now present the concept of a reducible face decomposition of 
a plane elementary bipartite graph with $n$ finite faces \cite{ZZ00},
which is built on a bipartite ear decomposition of an elementary bipartite graph \cite{LP09}.  
A \textit{bipartite ear decomposition} is defined as follows. 
Start from an edge $e$, join its two end vertices by a path $P_1$ of odd length 
to obtain an even cycle $G_1$,
proceed inductively to build a sequence of plane bipartite
graphs $G_i$ for $2 \le i \le n$ where $G_n=G$ as follows.
If $G_{i-1} = e + P_1 + \ldots + P_{i-1}$ has already been constructed, 
then $G_i=G_{i-1}+P_i$ can be obtained by adding the $i$th path $P_i$ of odd length  such that 
$P_i$ has exactly two end vertices of different colors in common with $G_{i-1}$, and $P_i$ is called the \textit{$i$th ear} of $G$.
A bipartite ear decomposition 
is called a \textit{reducible face decomposition}   if $G_1$ is the 
periphery of a finite face $s_1$ of $G$, and the $i$th ear $P_i$ lies in the exterior of $G_{i-1}$ 
such that $P_i$ and a part of the periphery of $G_{i-1}$
form the periphery of a finite face $s_i$ of $G$ for all $i \in \{2, \ldots, n \}$.
For such a decomposition, we use notation $\textrm{RFD}(G_1, G_2, \ldots, G_n)$, where $G_n=G$.
It was shown \cite{ZZ00} that a plane bipartite graph different from $K_2$ is elementary 
if and only if it has a reducible face decomposition
starting from any finite face of the graph.

Assume that $G$ is a plane bipartite graph whose vertices are properly colored black and white 
such that adjacent vertices receive different colors. Let $M$ be a perfect matching of $G$.
An $M$-alternating cycle $C$ of $G$ is \textit{proper $M$-alternating} (respectively, \textit{improper $M$-alternating})
if every edge of $C$ belonging to $M$ goes from white to black vertex 
(respectively, from black to white vertex) along the clockwise orientation of $C$.
A plane bipartite graph $G$ with a perfect matching has a unique perfect matching $M_{\hat{0}}$ (respectively, $M_{\hat{1}}$)
such that $G$ has no proper $M_{\hat{0}}$-alternating cycles 
(respectively, no improper $M_{\hat{1}}$-alternating cycles) \cite{ZZ97}.
If $G$ is a plane (weakly) elementary bipartite graph, then there is a finite distributive lattice structure 
on the set $\mathcal{M}(G)$
of all perfect matchings of $G$ \cite{LZ03}, and the $height(\mathcal{M}(G))$ is the distance between 
$M_{\hat{0}}$ and $M_{\hat{1}}$ in the resonance graph $R(G)$ \cite{ZLS08}.

\subsection{Resonance graphs}
The {\em resonance graph} (also called \textit{$Z$-transformation graph}) $R(G)$ of a plane bipartite graph $G$ is the graph 
whose vertices are the  perfect matchings of $G$, and two perfect matchings $M_1,M_2$ are adjacent 
if their symmetric difference $M_1 \oplus M_2$ forms the edge set of exactly one cycle 
that is the  periphery of a finite face $s$ of $G$ \cite{ZZY04}, 
and the edge $M_1M_2$ is said to have the \textit{face-label} $s$.
By definitions, we can see that if $G$ is a plane weakly elementary bipartite graph with elementary components $G_1, G_2, \ldots, G_t$,
then its resonance graph $R(G)$ is the Cartesian product $\Box_{i=1}^{t}R(G_i)$,  see a detail explanation provided in Section 2.3 \cite{BCTZ25}. 
For a  plane bipartite graph $G$ with a perfect matching, its resonance graph $R(G)$ is connected if and only if $G$ is weakly elementary \cite{ZZY04}.

In \cite{BCTZ25}, we introduced a new concept called peripherally 2-colorable graphs
to characterize when resonance graphs are daisy cubes. 
Let $G$ be a plane elementary bipartite graph different from $K_2$. 
Then $G$ is called \textit{peripherally 2-colorable} if  every vertex of $G$ has degree 2 or 3, 
vertices with degree 3 (if exist) are all exterior vertices of $G$,
and $G$ can be properly colored black and white so that two vertices with the same color are nonadjacent,
and vertices with degree 3 (if exist) are alternatively black and white along the clockwise orientation of the periphery of $G$. 
We showed that the resonance graph $R(G)$ of  a plane elementary bipartite graph $G$ different from $K_2$
is a daisy cube if and only if  $G$ is peripherally 2-colorable.  
Furthermore,   the resonance graph $R(G)$ of a  plane bipartite graph $G$ is a daisy cube if and only if $G$ is weakly elementary 
and every elementary component of $G$ different from $K_2$ is peripherally 2-colorable.
See Figure \ref{example_periph}.

\begin{figure}[h!] 
\begin{center}
\includegraphics[scale=0.7, trim=0cm 0.5cm 0cm 0cm]{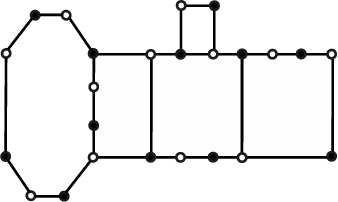}
\end{center}
\caption{\label{example_periph} A peripherally 2-colorable graph.}
\end{figure}

\subsection{Daisy cubes and median graphs}
Let $G$ be a graph with the vertex set $V(G)$ and the edge set $E(G)$.
The \textit{distance} $d_G(u,v)$ between two vertices $u$ and $v$ in $G$ 
is the length of a shortest path between two vertices in $G$.
The \textit{interval} $I_G(u,v)$ between two vertices $u$ and $v$ in  $G$ 
is the set of vertices located on shortest paths between two vertices in $G$.
 A subgraph of $G$ induced by a vertex subset $X \subseteq V(G)$ is denoted as $\langle X \rangle$. 
If $G$ is a graph whose vertex set is a poset $(V(G), \le)$, then 
the maximal elements of the poset $(V(G), \le)$  are called the \textit{maximal vertices} of $G$,
and the minimum element (if exists) of the poset $(V(G), \le)$ is called the \textit{minimum vertex} of $G$.

A \textit{Cartesian product} $\Box_{i=1}^{n} G_i$ of $n$ graphs $G_1, G_2, \ldots G_n$, {where $n \ge 2$,}
is a graph whose vertex set is $V(G_1) \times V(G_2) \times \cdots \times V(G_n)$ 
and two vertices $(x_1,\ldots, x_n)$ and $(y_1,\ldots, y_n)$ 
are adjacent in  $\Box_{i=1}^{n} G_i$ if there exists an $i \in \{1, \ldots, n \}$ such that   $x_iy_i$ is an edge of $G_i$, 
and $x_j=y_j$ for any $j \in \{1, \ldots, n \} \setminus \{i \}$.  
An \textit{$n$-dimensional hypercube} $Q_n$  is a Cartesian product $\Box_{i=1}^{n} K_2$, 
while $Q_0$ is a one-vertex graph, {and $Q_1$ is the one-edge graph.}
We say that a hypercube is \textit{nontrivial} if it is not a one-vertex graph.

An induced subgraph $H$ of $G$ is an \textit{isometric subgraph} if $d_H(u,v)=d_G(u,v)$ 
for every two vertices $u$ and $v$ of $H$. 
\textit{Partial cubes} are isometric subgraphs of hypercubes. 
\textit{Djokovi\' c-Winkler relation} (briefly, \textit{relation $\Theta$}) 
plays an important role in the study of partial cubes.
Two edges $x_1x_2$ and $y_1y_2$ of a connected graph $G$ are  
in \textit{relation $\Theta$} if $d_G(x_1, y_1) + d_G(x_2, y_2) \neq  d_G(x_1, y_2) + d_G(x_2, y_1)$.
The relation $\Theta$ is an equivalence relation on the edge set of any partial cube, see \cite{D73}.

Let $(\mathcal{B}^n, \le)$ be a poset on the set of all binary strings of length $n \ge 1$
with the partial order $u_1u_2 \ldots u_n \leq v_1v_2 \ldots v_n$ if $u_i \leq v_i$ for all $1 \le i \le n$.
A \textit{daisy cube} generated by $X  \subseteq \mathcal{B}^n$ is an induced  subgraph of $Q_n$, 
and defined  as $Q_n(X)=\langle \{u \in \mathcal{B}^{n} \ | \ u \leq x \textrm{for some } x \in X \} \rangle$.
 If  $\widehat{X}$ is the anti-chain consisting of the maximal elements of the poset $(X, \leq) \subseteq (\mathcal{B}^n, \le)$,
 then  $Q_n(X)=Q_n(\widehat{X})= \left\langle \cup_{x \in \widehat{X}} I_{Q_n}(x, 0^n)  \right\rangle$, see \cite{KM19}.
We call that $\widehat{X}$ is the \textit{set of maximal vertices} 
 of the daisy cube $Q_n(X)$, and $0^n$ is the \textit{minimum vertex} of the daisy cube $Q_n(X)$, see \cite{BCTZ25}.
By \cite{KM19}, daisy cubes are partial cubes. 
A binary coding of length $n$ for the vertex set of a graph $G$ isomorphic to a daisy cube is called a \textit{proper labelling}
if it induces an isometric embedding of $G$ into the hypercube $Q_n$  as a daisy cube \cite{T20}.
The \textit{isometric dimension} of a daisy cube $H$ is the minimum dimension of a hypercube that  
$H$ can be isometrically embedded into.

A graph $G$ is a \textit{median graph} if $G$
 is a connected graph such that for every triple vertices $u,v,w$ of $G$, 
$I_G(u,v) \cap I_G(u,w) \cap I_G(v,w)$ has a unique vertex, which is called the unique \textit{median} of three vertices \cite{M80}.
Median graphs form an important subfamily of partial cubes \cite{HIK11},
and resonance graphs of plane weakly elementary bipartite graphs are median graphs \cite{ZLS08}.
We observe that a median graph is not necessarily a daisy cube.
For example, a path $P_4$   is a median graph but not a daisy cube.
On the other hand, a daisy cube is not necessarily a median graph either.
For example, a gear graph (also called a bipartite wheel) \cite{K09} $BW_3$ is defined as 
a graph with a central vertex $w$ surrounded by a cycle $C_6$ 
such that $w$ is adjacent to alternating vertices on $C_6$.
Then $BW_3$ is a daisy cube with a proper labelling such that the binary string for $w$ is $000$,
and the binary strings for vertices of $C_6$ are listed $100, 110, 010, 011, 001, 101$ along the clockwise orientation of $C_6$.
The three vertices with binary strings $110, 011, 101$ do not have a unique median.
Hence, $BW_3$ is a daisy cube but not a median graph.

\section{Hypercubes of $R(G)$ and resonant sets of $G$}
 
We start from a relation between resonant sets of a plane weakly elementary bipartite graph
and hypercubes of its resonance graph, which
generalizes the main result (Theorem 2) in  \cite{SKG06} from a benzenoid
graph with a perfect matching to any plane  bipartite graph with a perfect matching.

\begin{lemma}\label{L:mapping-f}
Let $G$ be a plane bipartite graph with a perfect matching such that $G$ is different from $K_2$.
Let $\mathcal{H}(R(G))$ be the set of  nontrivial hypercubes of $R(G)$.
Let $\mathcal{RS}(G)$ be the set of  nonempty resonant sets of $G$.   
Let $k$ be a positive integer.
Define  $f: \mathcal{H}(R(G)) \rightarrow \mathcal{RS}(G)$ such that 
if $Q$ is a $k$-dimensional hypercube of $R(G)$, 
then $f(Q) = S_Q$ is a cardinality $k$ resonant set of $G$ with the property that 
each finite face in $S_Q$ is the face-label of a  $\Theta$-class of  $Q$.  Then $f$ is well-defined and surjective. 
Moreover,  $f^{-1}(S)$ is a unique $k$-dimensional hypercube of $R(G)$ 
only when the cardinality $k$ resonant set $S$ of $G$ is canonical.
\end{lemma}
\begin{proof}  It is well known \cite{HIK11} that any two edges on a shortest path cannot be in the same $\Theta$-class of a graph.
Moreover, for a plane bipartite graph $G$,  
antipodal edges of a 4-cycle in $R(G)$ have the same face-label,
and face-labels of two incident edges of a 4-cycle in $R(G)$ are vertex disjoint finite faces of $G$ \cite{ZOY09}.
It follows that edges incident to a common vertex of a hypercube in a graph are pairwise contained in different $\Theta$-classes
of the graph, and the face-labels of edges incident to a common vertex of a hypercube in $R(G)$ are pairwise vertex disjoint finite faces of $G$.

Let $k$ be a positive integer and $Q$ be a $k$-dimensional hypercube of $R(G)$.
Then the edge set of $Q$ consists of $k$  $\Theta$-classes 
such that each $\Theta$-class has a face-label, and the set of face-labels for the $\Theta$-classes of   $Q$
corresponds to a unique set $S_Q$ of $k$ vertex disjoint finite faces of $G$.
Let $S_Q=\{s_1, s_2, \ldots, s_k\}$. {We can see that $S_Q$ is a resonant set of $G$ as follows.
Let $M$ be an arbitrary perfect matching of $G$ that is a vertex of the hypercube $Q$ in $R(G)$.
We observe that there are $k$ edges of the hypercube incident to $M$ in $R(G)$,
 and face-labels of these $k$ edges are pairwise vertex disjoint finite faces of $G$.
 Hence, the set of face-labels of these $k$ edges incident to $M$ must be $S_Q$.
 It follows that  each finite face in $S_Q$ is $M$-resonant, and so $S_Q$ is a resonant set.}
Therefore,  the set of finite faces of $G$ corresponding to the set of face-labels of 
$\Theta$-classes of  $Q$ is  a unique cardinality $k$ resonant set $S_Q$ of $G$.
So, $f$ is well-defined.

Let $S$ be a  cardinality $k$ resonant set of $G$  for some positive integer $k$.
Then $S$ has $2^k$ perfect matchings because the periphery of each  finite face in $S$ is a cycle with two perfect matchings.
Let $S=\{s_1, s_2, \ldots, s_k\}$.
Let $\mathcal{M}(S)$ be the set of all perfect matchings of $S$. 
Then $|\mathcal{M}(S)|=2^k$ and any $\widehat{m} \in \mathcal{M}(S)$ 
has the property that all finite faces of $S$ are $\widehat{m}$-resonant.
Assume that the vertices of $G$ are properly colored black and white 
such that adjacent vertices receive different colors.
Then for each $\widehat{m} \in \mathcal{M}(S)$, we can assign $\widehat{m}$  
a binary string $l_{k}[\widehat{m}]$ of length $k$   
as $b_1b_2 \ldots b_k$ 
such that $b_i=0$ if $\partial s_i$ is proper $\widehat{m}$-alternating, and $b_i=1$ otherwise. 
Then $\{l_{k}[\widehat{m}] \mid \widehat{m} \in \mathcal{M}(S)\}$ induces a $k$-dimensional hypercube $Q_{|S|}$
such that any edge has two end vertices differing in exactly one position $i$ and corresponding to the finite face $s_i \in S$, 
for some $1 \le i \le k$.

Let $\mathcal{M}(G-S)$ be the set of all perfect matchings of $G-S$.
If $G-S$ is empty, then we assume that $\emptyset$ is a perfect matching of $G-S$.
Hence, we can see that each perfect matching $m$ in $\mathcal{M}(G-S)$ can be extended to a set of $2^k$ perfect matchings of $G$
whose restriction on $S$ is the set of $2^k$ perfect matchings $\mathcal{M}(S)$, 
which induces a $k$-dimensional hypercube $Q^m_{|S|}$ of $R(G)$ 
such that the set of face-labels of $\Theta$-classes of  $Q^m_{|S|}$ is $S$.
Hence, for each cardinality $k$ resonant set $S$ of $G$, 
there is a set $\mathcal{H}_S=\{Q^m_{|S|} \mid m \in \mathcal{M}(G-S)\}$ 
of $k$-dimensional hypercubes of $R(G)$ with the property that 
 for each hypercube $Q^m_{|S|}$ in $\mathcal{H}_S$,   
 $S$ is the set of face-labels of $\Theta$-classes of  $Q^m_{|S|}$.
It follows that $f^{-1} (S)=\mathcal{H}_S$, and so $f$ is surjective.

Note that $|\mathcal{H}_S|=|\mathcal{M}(G-S)|$. 
Then  $|f^{-1} (S)| =1$ only when $|\mathcal{M}(G-S)|=1$, 
that is, $G-S$ is either empty or has a unique perfect matching.
Hence,  $|f^{-1} (S)| =1$ only when $S$ is a canonical resonant set of $G$.
\end{proof}

\begin{remark}  Lemma \ref{L:mapping-f}  also follows by Theorem 1.5 in \cite{TY23} which provides
a bijection between the induced hypercubes in the resonance graph
$R(G)$ and the Clar covers of a graph $G$ on closed surfaces, where
a Clar cover is a spanning subgraph $S$ of $G$ such that every component of $S$ 
is either the boundary of an even face or an edge.
We have provided the proof of Lemma  \ref{L:mapping-f}  for completeness and simplicity since we only need 
to consider plane bipartite graphs in this paper.
\end{remark}
 
%\smallskip
 
%\smallskip
 
Recall  \cite{BCTZ25} that for a plane elementary bipartite graph $G$ different from $K_2$, 
$R(G)$ is a daisy cube if and only if $G$ is peripherally 2-colorable. 
By definition, if $G$ is peripherally 2-colorable, then the induced subgraph of $G$ 
obtained by removing vertices on the periphery of $G$ is either empty or has a unique perfect matching, 
and so the infinite face of $G$ is forcing.
We next consider  plane elementary bipartite graphs whose infinite face is forcing.

\begin{lemma}\label{L:InfiniteForcing-TwoNiceCycles}
 Let $G$ be a plane elementary bipartite graph different from $K_2$. 
Then the infinite face of $G$ is forcing if and only if any two vertex disjoint cycles 
whose union forms a nice subgraph of $G$ have disjoint interiors.
\end{lemma}
\begin{proof} \textit{Sufficiency.}    
Suppose that the infinite face of $G$ is not forcing.  
Then the resulted subgraph  $G'$ obtained by deleting all vertices 
on the periphery cycle $\partial G$ has at least two perfect matchings $m_1$ and $m_2$.
It is well known that the edges in the symmetric difference of $m_1$ and $m_2$ 
form a set $S$ of vertex disjoint $(m_1,m_2)$-alternating cycles contained in $G'$.
Let $M_1$ and $M_2$ be perfect matchings of $G$ extended from $m_1$ and $m_2$ respectively 
such that $\partial G$ is $(M_1, M_2)$-alternating. 
Then $S$ is also a set of vertex disjoint $(M_1,M_2)$-alternating cycles. 
Let $C$ be a cycle from $S$. Then the union of $C$ and $\partial G$ forms a nice subgraph of $G$,
since both $C$ and $\partial G$ are $(M_1, M_2)$-alternating. 
But the interior of $C$ is contained in the interior of $\partial G$, which is a contradiction.

\textit{Necessity.} 
Suppose that there are two vertex disjoint cycles $C_1$ and $C_2$ whose union
forms a nice subgraph of $G$ while their interiors are not disjoint.
Then $G$ has a perfect matching  $M$ such that both $C_1$ and $C_2$ are $M$-alternating.
Without loss of generality, we can assume that $C_1$ is contained in the interior of $C_2$. 
By Corollary 3.4 in \cite{ZZ00}, there exists a finite face $s_1$ of $G$ contained in 
the interior of $C_1$ such that $s_1$ is $M$-resonant and vertex disjoint from $C_2$.
For $i=1,2$, let $I[C_i]$ denote the induced plane subgraph of $G$ generated 
by all vertices either on $C_i$  or contained in the interior of $C_i$.
Note that the restriction of $M$ on $I[C_2]$ is a perfect matching $m$ of $I[C_2]$ such that both $C_1$ and $C_2$ are $m$-alternating.
It follows that the infinite face of $I[C_2]$ is not forcing.
By Theorem 2.9  in \cite{ZZ00}, each $I[C_i]$ for $i=1,2$ is elementary.
Assume that $I[C_1]$ contains $n_1$ finite faces and $I[C_2]$ contains $n_2$ finite faces.
Then $n_1< n_2$, and all finite faces of $I[C_1]$ are also finite faces of $I[C_2]$.
We can construct a RDF$(G_1, \ldots,  G_{n_1}, G_{n_1+1}, \ldots, G_{n_2}, 
G_{n_2+1}, \ldots, G_n)$  
of $G=G_n$ starting from $s_1$ with 
a sequence of  associated finite faces $s_1, \ldots, s_{n_1}, s_{n_1+1}, \ldots, s_{n_2}, s_{n_2+1}, \ldots, s_n$
such that  $I[C_1]$ has a RDF$(G_1, \ldots, G_{n_1})$ with 
a sequence of  associated finite faces $s_1, \ldots, s_{n_1}$,
and $I[C_2]$ has a RDF$(G_1, \ldots G_{n_1}, G_{n_1+1}, \ldots, G_{n_2})$ with 
a sequence of associated finite faces $s_1, \ldots, s_{n_1}, s_{n_1+1}, \ldots, s_{n_2}$.
Then by the proof of Theorem 2 in \cite{C21},  $height(\mathcal{M}(G_1)) =1$, 
and for $2 \le i \le n$, $height(\mathcal{M}(G_i)) \ge height(\mathcal{M}(G_{i-1}))+1$.
Moreover,  for each $2 \le i \le n$, $height(\mathcal{M}(G_i)) \ge i$ and the equality holds if and only if the infinite face of $G_i$ is forcing.
Recall that the infinite face of   $G_{n_2}=I[C_2]$ is not forcing. It follows that  $height(\mathcal{M}(G)) > n$, 
and so the infinite face of $G$ is not forcing.
\end{proof}
\smallskip

A \textit{coronene} is the benzenoid system shown in Figure \ref{coro}. 
It is known  \cite{ZC86} that  an elementary benzenoid system  $G$ has no coronenes as nice subgraphs 
if and only if  any two vertex disjoint cycles whose union forms a nice subgraph of $G$ have disjoint interiors.
The following corollary follows immediately by Lemma \ref{L:InfiniteForcing-TwoNiceCycles}.

\begin{figure}[h!] 
\begin{center}
\includegraphics[scale=0.7, trim=0cm 0.5cm 0cm 0cm]{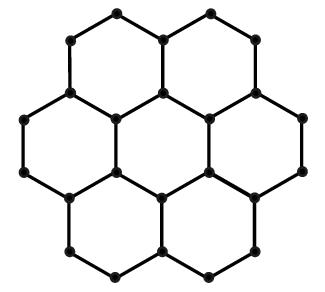}
\end{center}
\caption{\label{coro} A coronene.}
\end{figure}

\begin{corollary} \label{C:ForcingInfiniteFace-NoCoronenes} 
Let $G$ be an elementary benzenoid system. Then the infinite face of $G$ is forcing if
and only if $G$ has no coronenes as nice subgraphs.  
\end{corollary}

Next, we present a bijection between 
the set of maximal hypercubes of $R(G)$ and the set of maximal resonant sets of a plane elementary bipartite graph $G$  
whose infinite face is forcing,
which extends Lemma 4.2  and Theorem 4.3 in \cite{TZ12}
from an elementary benzenoid system without coronenes as nice subgraphs 
to a plane elementary bipartite graph whose infinite face is forcing.

\begin{theorem}\label{T:M-RS-CS-Hypercubes}
Let $G$ be a plane elementary bipartite graph whose infinite face is forcing.   
Then the following properties hold true.
%\begin{itemize}

$(i)$ Let $S$ be a resonant set of $G$. Then $S$ is maximal if and only if $S$ is canonical.

$(ii)$ Let $\mathcal{H}(R(G))$ be the set of  nontrivial hypercubes of $R(G)$, 
$\mathcal{RS}(G)$ be the set of  nonempty resonant sets of $G$, 
and $f: \mathcal{H}(R(G)) \rightarrow \mathcal{RS}(G)$ be the mapping defined in Lemma \ref{L:mapping-f}.
Then $f$ induces a bijection between the set of maximal hypercubes of $R(G)$ and
the set of  maximal resonant sets of $G$  which maps a maximal  hypercube $Q$ of $R(G)$
to  a maximal resonant set $S$ of $G$ such that $S$ is the set of face-labels of $\Theta$-classes of
 $Q$.
%\end{itemize}
\end{theorem}

\begin{proof}   $(i)$ \textit{Sufficiency.} Let $S$ be a canonical resonant set of $G$. 
Then $G-S$ is either empty or has a unique perfect matching. 
Hence, if $M$ is a perfect matching of $G$ such that each finite face in $S$ is $M$-resonant, 
then $G$ has no  other $M$-resonant faces
except the faces in $S$. 
It implies that $S$ cannot be contained in any other resonant set of $G$, and
so $S$ is maximal.

\textit{Necessity.} Let $S$ be a maximal resonant set of $G$. Suppose that $S$ is not canonical.
Then $G-S$ has at least two different perfect matchings $m_1$ and $m_2$.  
Since the symmetric difference of any two perfect matchings is a set of edges  of vertex disjoint cycles,
it follows that $G-S$ contains at least one $(m_1, m_2)$-alternating cycle $C$.
Let $M_1$ and $M_2$ be two perfect matchings of $G$ that are extended from $m_1$ and $m_2$, respectively,
such that all finite faces contained in $S$ are $(M_1,M_2)$-resonant.
Then the $(m_1, m_2)$-alternating cycle $C$ contained in $G-S$ is also an $(M_1,M_2)$-alternating cycle. 
Let $I[C]$ denote the induced plane subgraph of $G$ generated by all vertices either on $C$  or contained in the interior of $C$.
Then  by Corollary 3.4 in \cite{ZZ00}, 
there exists a finite face $s_0$ of $G$ that is contained in $I[C]$, and $s_0$ is $M_1$-resonant. 
If $S$ and $I[C]$ are vertex disjoint,   then $S \cup \{s_0\}$ is a resonant set. 
This is a contradiction to the assumption that $S$ is maximal.
If $S$ and $I[C]$ are not vertex disjoint, then there exists a finite face $s \in S$ contained in $I[C]$.
Note that $C$ cannot be the periphery of any finite face in $S$ since $C$ is contained in $G-S$.
It follows that the periphery of $s \in S$ and $C$ are two vertex disjoint cycles which form a nice subgraph of $G$,
but the interior of $s$ is contained in the interior of $C$.
Then by Lemma \ref{L:InfiniteForcing-TwoNiceCycles}, the infinite face of $G$ is not forcing. 
This is a contradiction.

$(ii)$  
Let $f: \mathcal{H}(R(G)) \rightarrow \mathcal{RS}(G)$ be a mapping defined in Lemma \ref{L:mapping-f},
where $\mathcal{H}(R(G))$ is the set of nontrivial  hypercubes of $R(G)$ 
and $\mathcal{RS}(G)$ is the set of  nonempty   resonant sets of $G$.
Let $S$ be a maximal resonant set of $G$ of cardinality $k$, where $k$ is a positive integer. 
By $(i)$, $S$ is canonical.  
Lemma \ref{L:mapping-f} then assures that  $f^{-1}(S)$ is a unique $k$-dimensional hypercube $Q$ 
from $\mathcal{H}(R(G))$ such that $S$ is the set of face-labels of $\Theta$-classes of  $f^{-1}(S)=Q$. 
If $Q$ is not a maximal hypercube of $R(G)$, then there exists a $k'$-dimensional hypercube $Q'$ of $R(G)$  
such that $k'>k$ and $Q$ is contained in $Q'$. It follows that $f(Q')=S'$ is a cardinality $k'$ resonant set of $G$ 
such that $S \subseteq S'$. This is a contradiction since $S$ is a maximal resonant set. 
We have shown that
if $S$ is a cardinality $k$ maximal resonant set of $G$, then $f^{-1}(S)$ is a $k$-dimensional maximal hypercube of $R(G)$
such that $S$ is the set of face-labels of $\Theta$-classes of  $f^{-1}(S)$.

Next, we prove that if $Q$ is a $k$-dimensional maximal hypercube of $R(G)$, 
then $f(Q)=S$ is a cardinality $k$ maximal resonant set  of $G$ consisting of the face-labels of the $\Theta$-classes of $Q$.
Let $S=\{s_1, s_2, \ldots, s_k\}$. 
Let $\mathcal{M}(G-S)$ be the set of all perfect matchings of $G-S$. 
By the proof of Lemma \ref{L:mapping-f}, we can let $\mathcal{H}_{S}=\{Q^m_{|S|} \mid m \in \mathcal{M}(G-S)\}$ be the set 
of $k$-dimensional hypercubes in $R(G)$ with the property that 
 for each hypercube $Q^m_{|S|}$ in $\mathcal{H}_{S}$,   $S$ is the set of face-labels of $\Theta$-classes of  $Q^m_{|S|}$.
Then   $Q=Q^{m_1}_{|S|}$ for some $m_1 \in \mathcal{M}(G-S)$.
So, any vertex of $Q$ is a perfect matching $M$ of $G$  
such that  $M \cap E(G-S)=m_1$ and all finite faces in $S$ are $M$-resonant.
 
Suppose that $S$ is not maximal. Then $S$ is not canonical by $(i)$.
Then $G-S$ has at least one other perfect matching $m_2$ that is  different from $m_1$.
Let  $M_1$ be a perfect matching of $G$ such that $M_1 \cap E(G-S)=m_1$ and all finite faces in $S$ are $M_1$-resonant.
Let $M_2$ be a perfect matching of $G$ such that $M_2 \cap E(G-S)=m_2$ and 
the restriction of $M_2$ on $S$ is identical to the restriction of $M_1$ on $S$.
Then $M_1 \oplus M_2=m_1 \oplus m_2$ is a set of $(M_1,M_2)$-alternating cycles contained in $G-S$.
Let $C $ be an $(M_1,M_2)$-alternating cycle that lies in $G-S$ and
let $I[C]$ be the induced plane subgraph of $G$ generated by all vertices on $C$ or in the interior of $C$.
By Corollary 3.4 \cite{ZZ00}, there is a finite face $s_{k+1}$ of $G$ contained in $I[C]$ that is also $M_1$-resonant.  It is clear that $s_{k+1}$ is also $m_1$-resonant.
Recall that the infinite face of $G$ is forcing. 
If $C$ is vertex disjoint with the periphery of $s_{k+1}$, then the union $C$ and $\partial s_{k+1}$ forms a nice subgraph of $G$, 
while interior of $s_{k+1}$ is contained in the interior of $C$. This is a contradiction to the conclusion of Lemma \ref{L:InfiniteForcing-TwoNiceCycles}.
Therefore, $C$ must have nonempty intersection with $s_{k+1}$.
Moreover, $s_{k+1} \notin S$ since $C$ belongs to $G-S$. Let $S'=S \cup \{s_{k+1}\}$. 
Recall that $m_1 \in \mathcal{M}(G-S)$ and $s_{k+1}$ is $m_1$-resonant.
Let $m'_2=m_1 \oplus E(s_{k+1})$. Then $m'_2 \in \mathcal{M}(G-S)$. (It is possible that $m'_2=m_2$.)
 
{Let $Q'=Q^{m'_2}_{|S|} \in \mathcal{H}_S$.  Then any vertex of $Q'$ is a perfect matching $M'$ of $G$  
such that $M' \cap E(G-S)=m'_2$ and all finite faces in $S$ are $M'$-resonant. 
Note that $m_1 \oplus m'_2=E(s_{k+1})$.
Recall that any vertex of $Q$ is a perfect matching $M$ of $G$  
such that  $M \cap E(G-S)=m_1$ and all finite faces in $S$ are $M$-resonant.
Hence, for each vertex $M$ of the hypercube $Q$,   
there is a unique vertex  $M'$ of  the hypercube $Q'$
such that $M' \cap E(G-S)=m'_2$ and the restriction of $M'$ on $S$ is identical to the restriction of $M$ on $S$.
It follows that $M \oplus M'=m_1 \oplus m'_2= E(s_{k+1})$.
By definition, $MM'$ is an edge of $R(G)$ with the face-label $s_{k+1}$.
Note that the symmetric difference of any two perfect matchings that are vertices in 
the hypercube $Q$ contains the edges of the periphery of least one face from $S$, and $s_{k+1}$ 
is vertex disjoint from any finite face from $S$.
It follows that  $M'$ is the unique vertex of $Q'$ that is adjacent to the given vertex $M$ of $Q$ in $R(G)$,
since two perfect matchings are adjacent in $R(G)$ if their symmetric difference 
is the periphery of exactly one finite face of $G$.
Let $V(Q)$ and $V(Q')$ be the vertex set of $Q$ and $Q'$ respectively.
Then there is a matching between $V(Q)$ and $V(Q')$ in $R(G)$ such that each edge of the matching has the face-label $s_{k+1}$.
Hence, the induced subgraph of $R(G)$ on $V(Q) \cup V(Q')$ is a $(k+1)$-dimensional hypercube of $R(G)$ which contains $Q$.
This is a contradiction to the assumption that $Q$ is maximal.}
 
Therefore,  $f$ induces a bijection between the set of maximal hypercubes of $R(G)$ 
and the set of  maximal resonant sets of $G$  which maps a   maximal  hypercube $Q$ of $R(G)$
to a maximal resonant set $S$ of $G$ such that $S$ is the set of face-labels of $\Theta$-classes of
 $Q$.  
 \end{proof}

In Figure \ref{res_graph} we can see a plane elementary bipartite graph $G$ whose infinite face is forcing. Denote by $s_1,s_2,s_3,s_4,s_5,s_6$ the finite faces of $G$. Then the maximal resonant sets of $G$ are $\{s_1,s_3,s_5 \}$, $\{s_1,s_4 \}$, $\{s_2,s_4 \}$, $\{s_2,s_5 \}$, and $\{s_6 \}$. On the other hand, these maximal resonant sets correspond to maximal hypercubes of the resonance graph $R(G)$. More precisely, $R(G)$ has one 3-dimensional maximal hypercube, three $2$-dimensional maximal hypercubes, and one  1-dimensional maximal hypercube.
 \bigskip
 \begin{figure}[h!] 
\centering
\includegraphics[scale=0.6, trim=0cm 0cm 0cm 0cm]{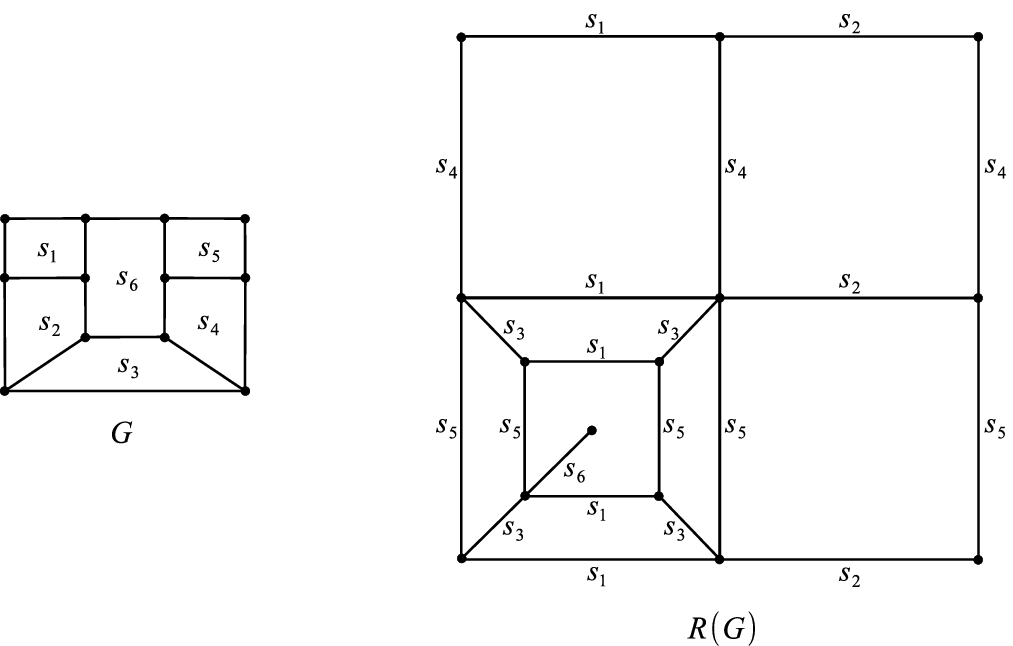}
\caption{A plane elementary bipartite graph $G$ and its resonance graph $R(G)$.} \label{res_graph} 
\end{figure}

In the following lemma we provide a connection between (maximal) resonant sets of a  peripherally 2-colorable graph $G$
and (maximal) independent sets of its inner dual $G^*$. 
We omit the trivial proof which follows immediately by the definitions. 
We also observe that it can be deduced easily from Theorem 3 in \cite{TZ26}.

\begin{lemma}\label{L:mapping-g}
Let $G$ be a peripherally 2-colorable graph and $G^*$ be the inner dual of $G$.
Let  $\mathcal{RS}(G)$ be the set of  nonempty  resonant sets of  $G$,
and $\mathcal{IS}(G^*)$ be the set of  nonempty  independent sets of $G^*$.  
Define $g: \mathcal{RS}(G) \rightarrow \mathcal{IS}(G^*)$  such that $g(S)=S^*$, 
where $S$ is a resonant set of $G$, and $S^*$ is the corresponding independent set of $G^*$.
 Then $g$ is a well-defined bijection which maps a resonant set $S$ of $G$
to an independent set $S^*$ of $G^*$ where $|S|=|S^*|=k$ for some positive integer $k$. 
Moreover, $S$ is a maximal resonant set of $G$ if and only if $S^*$ is a maximal independent set of $G^*$.
\end{lemma}
 
Recall that the infinite face of  any peripherally 2-colorable graph is forcing.
Based on Theorem \ref{T:M-RS-CS-Hypercubes} and 
 Lemma \ref{L:mapping-g}, 
if $G$ is a peripherally 2-colorable graph, then
we can establish a bijection between the set of maximal hypercubes  of 
its resonance graph $R(G)$ and the set of maximal independent sets of its inner dual $G^*$.

\begin{corollary}\label{C:HypercubeR(G)-IndependentSet(G*)} 
Let $G$ be a  peripherally 2-colorable graph  and $G^*$ be the inner dual of $G$.   
Let $f: \mathcal{H}(R(G)) \rightarrow \mathcal{RS}(G)$ be a mapping defined in Lemma \ref{L:mapping-f} 
and $g: \mathcal{RS}(G) \rightarrow \mathcal{IS}(G^*)$   a mapping defined in Lemma \ref{L:mapping-g}.
Define $h: \mathcal{H}(R(G)) \rightarrow \mathcal{IS}(G^*)$  such that $h(Q) = g(f(Q))$ for any $Q \in \mathcal{H}(R(G))$.
Then $h$ induces a bijection between the set of maximal hypercubes of $R(G)$ and the set of maximal independent sets of $G^*$
such that $h(Q_S)=S_Q^*$ if $Q_S$ is a $k$-dimensional maximal hypercube of $R(G)$,
and $S_Q^*$ is a cardinality $k$ maximal  independent set  of $G^*$  with the property that each vertex of $S_Q^*$
 corresponds to the  face-label of a unique $\Theta$-class of $Q_S$, where $k$ is a positive integer.
\end{corollary}

\begin{remark}{Corollary \ref{C:HypercubeR(G)-IndependentSet(G*)} can be obtained directly 
without Theorem \ref{T:M-RS-CS-Hypercubes} and Lemma \ref{L:mapping-g},
just using the nice structural properties of a peripherally 2-colorable graph and its resonance graph which is a daisy cube.}
Moreover, Corollary \ref{C:HypercubeR(G)-IndependentSet(G*)}  
can be generalized to plane weakly elementary bipartite graphs  whose each elementary component 
 different from $K_2$ is peripherally 2-colorable.  
\end{remark}

 \section{Daisy cubes generated by the set of maximal independent sets}
 
In this section, we characterize resonance graphs that are daisy cubes through independent sets. 
First, we show that it is possible to construct a daisy cube by using independent sets of an arbitrary graph.
 
\begin{definition}\label{D:Simplex}
Let $H$ be a graph with vertex set $V(H)=\{v_1, v_2, \ldots, v_n \}$.
For any subset $A \subseteq V(H)$,
define $\chi(A)=\chi_{A}(v_1)\chi_{A}(v_2) \ldots \chi_{A}(v_n)$ 
such that $\chi_{A}(v_j)=1$
if $v_j \in A$, and $\chi_{A}(v_j)=0$ otherwise.
Let $\mathcal{I}(H)=\{I_1, I_2, \ldots, I_t\}$ be the set of maximal independent sets of  $H$. 
Define $D_{\mathcal{I}} (H)$ as an induced subgraph  of
$Q_n$ whose vertex set is the downward closed subset $\{u \in \mathcal{B}^n \mid u \le \chi(I_i)  \textrm{ for some } I_i \in \mathcal{I}(H) \}$
of poset $(\mathcal{B}^n, \le)$, that is,
\[D_{\mathcal{I}} (H)=\langle \{u \in \mathcal{B}^n \mid u \le \chi(I_i)  \textrm{ for some } I_i \in \mathcal{I}(H) \} \rangle.\] 
\end{definition}

\begin{lemma} \label{L:DaisyCubeIndependentSets} 
For any graph $H$, $D_{\mathcal{I}} (H)$ is a daisy cube and isomorphic to the simplex graph $\mathcal{K}(\bar{H})$.
Moreover, there is a 1--1 correspondence between the set of vertices of $D_{\mathcal{I}} (H)$ 
and the set of independent sets of $H$ (including the empty set)
such that for each vertex $u$ of $D_{\mathcal{I}} (H)$, there is a  unique independent set $I_u$ of $H$ satisfying $u=\chi(I_u)$.
\end{lemma}
\begin{proof}   By Definition \ref{D:Simplex}, it is easy to see that $D_{\mathcal{I}} (H)$ is a daisy cube 
with the set of maximal vertices $\{\chi(I_1), \chi(I_2), \ldots, \chi(I_t)\}$, 
where $\{I_1, I_2, \ldots, I_t\}$ is the set of maximal independent sets of $H$.
It is clear that $\chi(\emptyset)=0^n$.  Moreover, each nonempty independent set $A$ of $H$  
defines a unique vertex $\chi(A)$ of $D_{\mathcal{I}} (H)$ different from $0^n$, 
and each vertex of $D_{\mathcal{I}} (H)$ different from $0^n$
is $\chi(A)$ for a unique nonempty independent set $A$ of $H$.
Therefore, there is a 1--1 correspondence between the set of vertices of $D_{\mathcal{I}} (H)$ 
and the set of independent sets of $H$ (including the empty set) 
such that for each vertex $u$ of $D_{\mathcal{I}} (H)$, there is a  unique independent set $I_u$ of $H$ satisfying $u=\chi(I_u)$.
Recall that $K$ is an independent set of $H$ if and only if $\bar{K}$ is a clique of $\overline{H}$. 
Also, binary codes of two independent sets are adjacent in $D_{\mathcal{I}} (H)$ if and only if they differ in exactly one vertex. 
It follows that $D_{\mathcal{I}} (H)$ is isomorphic to the simplex graph $\mathcal{K}(\bar{H})$.
\end{proof}

\begin{remark}
By the above lemma, we can see that every simplex graph is a daisy cube. 
{This conclusion
can be followed by the fact  \cite{BV89} that every simplex graph is a median graph,
and a recent result in \cite{xie} that a median graph $G$ is a simplex graph if and only if $G$ is a daisy cube.}
For the sake of completeness, we nevertheless include a proof of Lemma~\ref{L:DaisyCubeIndependentSets}, 
as the underlying construction will be used in the subsequent results.

{On the other hand, we observed that not every daisy cube is a median graph.  
Since every simplex graph is a median graph,} there exist daisy cubes that are not  
simplex graphs and cannot be represented as $D_{\mathcal{I}}(H)$
for any graph $H$.
\end{remark}

A characterization on when the resonance graph of a plane elementary bipartite graph $G$ is a daisy cube 
can be obtained in terms of the daisy cube generated by the set of maximal independent sets of its inner dual $G^*$.

\begin{lemma}\label{L:DaisyCube-ResonantGraph-MIS}
Let $G$ be a plane elementary bipartite graph  and $G^*$ be the inner dual of $G$.
Then $R(G)$ is  a daisy cube if and only if $R(G)$ is isomorphic to $D_{\mathcal{I}} (G^*)$
(or, $\mathcal{K}(\bar{G^*})$ equivalently), 
where $G^*$ is a tree.
\end{lemma} 
\begin{proof} Let $n$ be the number of finite faces of $G$.
If $n=0$, then $G$ is $K_2$, and the resonance graph $R(G)$  is the  one-vertex graph. 
On the other hand, the empty set is the only independent set of $G^*$,  
which means that $D_{\mathcal{I}} (G^*)$ is also the  one-vertex graph. 
In the following we assume that $n>0$, and so $G$ is a 2-connected graph.

Sufficiency follows by Lemma \ref{L:DaisyCubeIndependentSets}. It remains to show necessity.
Suppose that $R(G)$ is a daisy cube. Then by Corollary 3.6 in \cite{BCTZ25}, 
$G$ is a peripherally 2-colorable graph and  the isometric dimension of $R(G)$ is $n$. 
Then we can assume that the vertices of $R(G)$ are binary strings of length $n$.
By the proof of Theorem 3.5 in \cite{BCTZ25}, a peripherally 2-colorable graph $G$ is either a 2-connected outerplane bipartite
graph, or can be transformed into a 2-connected outerplane bipartite graph $G'$ that is also peripherally 2-colorable such that
 the inner duals of $G$ and $G'$ are the same. It is well known that a 2-connected outerplane bipartite graph is a tree,
 and so the inner dual of a peripherally 2-colorable graph $G$ is a tree.
    
By Corollary \ref{C:HypercubeR(G)-IndependentSet(G*)}, there is a bijection 
between the set of maximal hypercubes of $R(G)$ and the set of maximal independent sets of $G^*$,
which maps a maximal hypercube of $R(G)$ to a maximal  independent set  of $G^*$
that corresponds to the set of the face-labels of the $\Theta$-classes of the maximal hypercube.

Let $\{s^*_1, s^*_2, \ldots, s^*_n\}$ be the vertex set of $G^*$, 
where $s^*_j$ corresponds to the finite face $s_j$ of $G$ for $j \in \{1, 2, \ldots, n\}$.
Let  $\{I^*_1, I^*_2, \ldots, I^*_t\}$ be the set of maximal independent sets of  $G^*$.
Then we can use $\{Q_{|I^*_1|}, Q_{|I^*_2|}, \ldots, Q_{|I^*_t|}\}$ to represent the set of maximal hypercubes of $R(G)$,
where the maximal independent set $I^*_i$ of $G^*$ corresponds to the set of face-labels of the $\Theta$-classes of 
the maximal hypercube $Q_{|I^*_i|}$ in $R(G)$.

For  each maximal independent set $I^*_i$ of $G^*$,
define  $\chi(I_i^*) \in \mathcal{B}^n$ such that the $j$th position 
of $\chi(I_i^*)$ is $1$ if $s^*_j \in I^*_i$,  and $0$ otherwise.
It follows that the corresponding maximal hypercube of $R(G)$
is $Q_{|I^*_i|}=\langle \{u \in \mathcal{B}^n \mid u \leq \chi(I_i^*) \} \rangle$.
Recall that $R(G)$ is a daisy cube with the set of maximal hypercubes $\{Q_{|I^*_1|}, Q_{|I^*_2|}, \ldots, Q_{|I^*_t|}\}$.
Then $R(G)=\langle \cup_{i=1}^{t} \{u \in \mathcal{B}^n \mid u \leq \chi(I_i^*) \}  \rangle$,
where $\{\chi(I^*_1), \chi(I^*_2), \ldots, \chi(I^*_t)\}$ is the set of maximal vertices of $R(G)$.

Since $\{I^*_1, I^*_2, \ldots, I^*_t\}$ is the set of maximal independent sets of  $G^*$,
by the proof of Lemma \ref{L:DaisyCubeIndependentSets},
$\{\chi(I^*_1), \chi(I^*_2), \ldots, \chi(I^*_t)\}$ is the set of maximal vertices of $D_{\mathcal{I}} (G^*)$.
Consequently, daisy cubes $R(G)$ and $D_{\mathcal{I}} (G^*)$ have the same set of maximal vertices, which implies that they are isomorphic. 
\end{proof}

See Figure \ref{fig2} for the illustration of Corollary \ref{C:HypercubeR(G)-IndependentSet(G*)} 
and Lemma \ref {L:DaisyCube-ResonantGraph-MIS}.

\begin{figure}[h!] 
\begin{center}
\includegraphics[scale=0.7, trim=0cm 0.5cm 0cm 0cm]{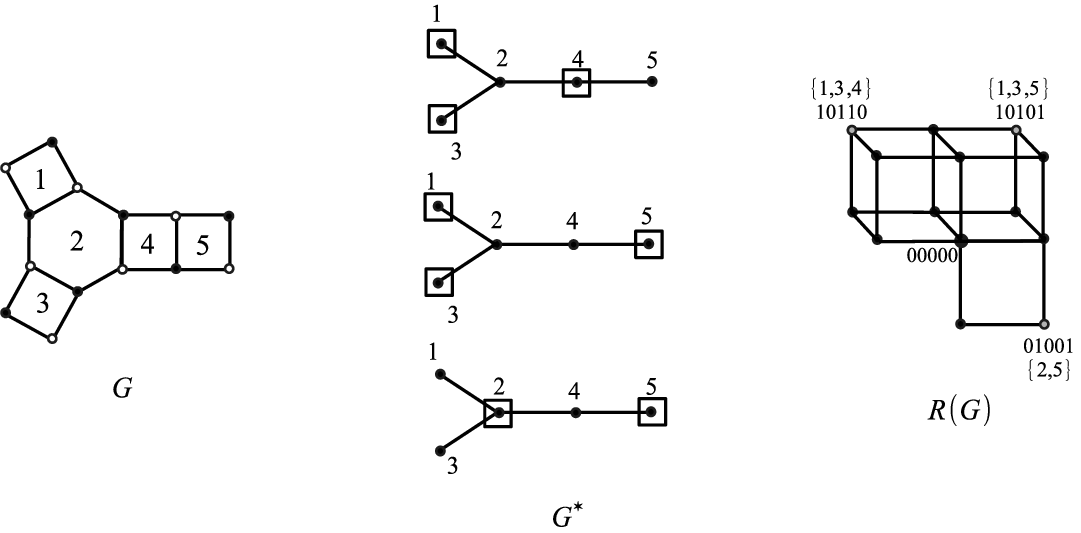}
\end{center}
\caption{\label{fig2} {A peripherally 2-colorable graph $G$ whose inner dual $G^*$ has three maximal independent sets
which correspond to the three maximal {vertices} of $R(G)$.}}
\end{figure}

We now present  a generalized version of Lemma \ref{L:DaisyCube-ResonantGraph-MIS}.
\begin{theorem}\label{T:General-DaisyCube-ResonantGraph-MIS}
Let $G$ be  a plane weakly elementary bipartite graph.
Let $G_a$ be the subgraph of $G$ obtained by removing all forbidden edges 
from $G$ and $G^*_a$ be the inner dual of $G_a$. 
Then  $R(G)$ is  a daisy cube if and only if $R(G)$ is isomorphic to $D_{\mathcal{I}} (G^*_a)$, 
(or, $\mathcal{K}(\bar{G^*_a})$ equivalently), where $G^*_a$ is a forest.
\end{theorem}
\begin{proof} Note that $G_a$ is a disjoint union of elementary components of $G$, 
which are denoted by $G_1, G_2, \ldots, G_t$.
By \cite{BCTZ25}, $R(G)=\Box_{i=1}^{t} R(G_i)$ and $R(G)$ is a daisy cube if and only if $R(G_i)$ is  a daisy cube for all $1 \le i \le t$.  
For $1 \le i \le t$,  by Lemma \ref{L:DaisyCube-ResonantGraph-MIS},
 $R(G_i)$ is  a daisy cube  if and only if $R(G_i)$ is isomorphic to $D_{\mathcal{I}} (G^*_i)$, where $G^*_i$ is 
{a tree as} 
 the inner dual of $G_i$. 
  Therefore, $R(G)$ is  a daisy cube if and only if $R(G)$ is isomorphic to $\Box_{i=1}^{t}D_{\mathcal{I}} (G^*_i)$.
 By the proof of Lemma \ref{L:DaisyCubeIndependentSets},
$D_{\mathcal{I}} (G^*_i)$ is a daisy cube such that
 $\chi(I)$ is a maximal vertex of $D_{\mathcal{I}} (G^*_i)$  if and only if $I$ is a maximal independent set  of $G^*_i$.
 By the proof of Theorem 4.1 in \cite{BCTZ25}, we can see that $\Box_{i=1}^{t}D_{\mathcal{I}} (G^*_i)$ is a daisy cube  
such that any maximal vertex of $\Box_{i=1}^{t}D_{\mathcal{I}} (G^*_i)$ is of the form  $\chi(I_1)\chi(I_2)\ldots\chi(I_t)$
 where $I_i$ is a maximal independent set of $G^*_i$ for $1 \le i \le t$.
 
 Let $G^*_a$ be the inner dual of $G_a$. {Then $G^*_a$ is a forest if and only $G^*_i$ is a tree for $1 \le i \le t$.}
By Lemma \ref{L:DaisyCubeIndependentSets}, 
$D_{\mathcal{I}} (G^*_a)$ is a daisy cube such that
 $\chi(I_a)$ is a maximal vertex of $D_{\mathcal{I}} (G^*_a)$  
 if and only if $I_a$ is a maximal independent set  of $G^*_a$.
Note that any maximal independent set of $G^*_a$ is a disjoint union 
of maximal independent sets of $G^*_1, G^*_2, \ldots, G^*_t$. 
It follows that any maximal vertex of $D_{\mathcal{I}} (G^*_a)$ is  $\chi(I_1 \cup I_2 \cup \ldots \cup I_t)$,
where $I_i$ is a maximal independent set of $G^*_i$ for $1 \le i \le t$.
By definition, we can see that $\chi(I_1)\chi(I_2)\ldots\chi(I_t)=\chi(I_1 \cup I_2 \cup \cdots \cup I_t)$.
Then two daisy cubes $\Box_{i=1}^{t}D_{\mathcal{I}} (G^*_i)$ and $D_{\mathcal{I}} (G^*_a)$ are isomorphic
since they have the same set of maximal vertices. 

Therefore, $R(G)$ is  a daisy cube if and only if $R(G)$ is isomorphic to $D_{\mathcal{I}} (G^*_a)$,
where   $G^*_a$ is a forest.
\end{proof}

We conclude this section with a corollary of Theorem \ref{T:General-DaisyCube-ResonantGraph-MIS}.

\begin{corollary}\label{C:ProperLabellingIndependentSets}
Let $G$ be  a plane weakly elementary bipartite graph whose elementary components 
with more than two vertices are peripherally 2-colorable. 
Let $G_a$ be the subgraph of $G$ obtained by removing all forbidden edges 
from $G$, and $G^*_a$ be the inner dual of $G_a$. 
Then there exists a bijection between the set of all perfect matchings of $G$ and the set of all independent sets of $G^*_a$ (including the empty set).
In particular, if $G$ is  a peripherally 2-colorable graph and $G^*$ is the inner dual of $G$,
then there exists a bijection between the set of all perfect matchings of $G$ and the set of all independent sets of  $G^*$ (including the empty set).
\end{corollary}
\begin{proof} By \cite{BCTZ25}, $R(G)$ is a daisy cube. 
By Theorem \ref{T:General-DaisyCube-ResonantGraph-MIS}, $R(G)$ is isomorphic to $D_{\mathcal{I}}(G^*_a)$.
Note that  the vertex set of $R(G)$ is the set of all perfect matchings of $G$.
By Lemma \ref{L:DaisyCubeIndependentSets}, there is a bijection between
 the vertex set of $D_{\mathcal{I}}(G^*_a)$ and the set of all independent sets of $G^*_a$ (including the empty set).
Therefore, there exists a bijection between the set of all perfect matchings of $G$ and the set of all independent sets of $G^*_a$. 
In particular, if $G$ is  a peripherally 2-colorable graph, then $G^*_a$ is $G^*$, and so the conclusion follows.
\end{proof}

 Note that one specific bijection between the set of all perfect matchings of a peripherally 2-colorable graph $G$ 
 and the set of all independent sets of $G^*$ was described in \cite{TZ26}.
{Based on a decomposition structure on a peripherally 2-colorable graph, 
Che and Chen \cite{CC25} gave an  algorithm
to produce a proper labelling for the vertex set of the resonance graph
of a peripherally 2-colorable graph.}
Another algorithm  
could be obtained by 
Lemma \ref{L:DaisyCube-ResonantGraph-MIS}, but it is {not} efficient since
we need to {list all independent sets after finding maximal independent sets} of a tree, and
Wilf  \cite{W86} showed that for a tree of order $n \ge 1$, 
the largest number of its maximal independent sets  is $2^{n/2-1}+1$ or $2^{(n-1)/2}$ 
depending on whether $n$ is even or odd.

 \section{Some applications}

Fibonacci cubes and Lucas cubes are well-studied special types of daisy cubes.
Fibonacci cubes $\Gamma_n$ ($n \ge 1$) are the resonance graphs of fibonaccenes, i.e., zigzag hexagonal chains \cite{KZ05},
but Lucas cubes $\Lambda_n$ ($n \ge 3$) cannot be resonance graphs of plane elementary bipartite graphs \cite{ZOY09}.
We present the following known properties of {Fibonacci cubes} as examples by applying our main results.
\smallskip

\begin{corollary}\label{C:FibonacciCubes} Let $\Gamma_n$ be a Fibonacci cube, where $n \ge 1$. Then

$(i)$  $\Gamma_n$ is isomorphic to $D_{\mathcal{I}} (P_n)$ (or, $\mathcal{K}(\bar{P_n})$), 
where $P_n$ is a path on $n$ vertices.

$(ii)$  For any $n \ge 1$, the number of maximal hypercubes of $\Gamma_n$ 
is equal to the Padovan number $a_{n}$, which is defined as 
$a_0=a_1=1$, $a_2=2$, and  $a_n = a_{n-2} + a_{n-3}$ for $n \geq 3$.

$(iii)$ For $\lceil n/3 \rceil \le k \le \lfloor (n+1)/2 \rfloor$, the number of maximal $k$-dimensional
hypercubes in  $\Gamma_n$ is equal to ${k+1 \choose n+1-2k}$.
\end{corollary}
\proof
It was shown in \cite{ZOY09} that the resonance graph $R(G')$
of a plane bipartite graph $G'$ is a Fibonacci cube $\Gamma_n$ if and only if $G'$
is weakly elementary with exactly one elementary component $G$ different from $K_2$,
and $G$ has a RFD$(G_1, G_2, \ldots, G_n)$, 
associated with finite faces $s_i (1 \le i \le n)$ and the odd length ears $P_i (2 \le i \le n)$,
where $P_i$ starts at $u_i$ and ends at $w_i$ along the clockwise orientation of $G_i$ for $2 \le i \le n$
such that $u_i$ of $P_i$ and $u_{i+1}$ of $P_{i+1}$ are in different colors, and 
two end vertices $u_{i+1}$ and $w_{i+1}$ of $P_{i+1}$ are internal vertices of $P_i$ for $2 \le i \le n-1$. 
We observe that this is equivalent to say that $G$ 
is a peripherally 2-colorable graph such that its inner dual $G^*$  is a path $P_n$ on $n$ vertices.
Hence,  property (i) holds true by Lemma \ref{L:DaisyCube-ResonantGraph-MIS}.

The number of maximal independent sets in $P_{n}$ is given by the Padovan sequence $a_0,a_1, a_2, \ldots, a_n, \ldots$, 
which is defined as 
$a_0=a_1=1$, {$a_2=2$}, and  $a_n = a_{n-2} + a_{n-3}$ for $n \geq 3$, see  \cite{F87}.
The number of maximal independent sets of size $k$ in $P_n$ is ${k+1 \choose n+1-2k}$ for
$\lceil n/3 \rceil \le k \le \lfloor (n+1)/2 \rfloor$, 
where $\lceil n/3 \rceil$ (respectively, $\lfloor (n+1)/2 \rfloor$) is the minimum (respectively, the maximum) possible size 
of a maximal independent set in $P_{n}$, see \cite{HS84}.
Hence,  properties (ii) and (iii) follow by Corollary \ref{C:HypercubeR(G)-IndependentSet(G*)}.
\qed\\

\begin{remark} Corollary \ref{C:FibonacciCubes} (i)  
is equivalent to the fact that a Fibonacci cube $\Gamma_n$ is $\mathcal{K}(\bar{P}_n)$ given in \cite{EKM23},
and Corollary \ref{C:FibonacciCubes} (iii) is a main result in \cite{M12} with different proofs.
\end{remark}

\begin{remark}
{Though  $\Lambda_n=\mathcal{K}(\bar{C}_n)$ \cite{MPZ01}, it is not clear to us if
$\Lambda_n$ is the simplex graph of the complement of a forest or not.
Hence, it cannot be concluded trivially by 
Theorem \ref{T:General-DaisyCube-ResonantGraph-MIS}
that Lucas cubes $\Lambda_n$ ($n \ge 3$) cannot be the resonance graphs of plane bipartite graphs.}
\end{remark}

By Lemma \ref{L:DaisyCube-ResonantGraph-MIS}, we know that
any daisy cube which is the resonance graph of a plane elementary bipartite graph 
is isomorphic to $D_{\mathcal{I}}(T)$ (or, $\mathcal{K}(\bar{T})$) for some tree $T$.
Searching  for all the maximal independent sets (or, maximal cliques) 
of a graph efficiently is fundamental in the theory of graphs and its applications \cite{TIAS77}.
This motivates us to characterize trees with $l$ maximal independent sets and determine daisy cubes 
with $l$ maximal vertices that are simplex graphs of the complements of trees. 
We  solve the above problem when $l$ is at most $5$.

For each integer $n \ge 2$, let $S^1_{n-1}$  be a star on $n$ vertices.
For each integer  $i \geq 2$,
let $S_{p,q}^i$ be a tree obtained from a path $P_i$ on $i$ 
vertices by adding $p$ pendant vertices incident to one end vertex of $P_i$ 
and $q$ pendant vertices incident to the other end vertex of $P_i$. 
If $i=2$, the obtained tree is called a \textit{bistar}.
Moreover, by $S_{p,q,r}^3$ we denote a tree obtained from a path 
$P =v_1v_2v_3$ on 3 vertices by adding $p$ pendant vertices incident to $v_1$, 
$q$ pendant vertices incident to $v_3$, and $r$ pendant vertices incident to $v_2$.

 Let $u$ be a vertex of a graph $G$. The set of vertices adjacent to $u$ in $G$ 
 is called the \textit{neighborhood} of $u$ in $G$, and denoted by $N_G(u)$.

\begin{lemma}\label{L:MIS(T)}
Let $T$ be a tree and $\mis(T)$ be the set of maximal independent sets of $T$. 
Then 
\begin{enumerate}
\item $|\mis(T)|=1$ if and only if $T$ is a one-vertex graph.

\item $|\mis(T)|=2$ if and only if $T$ is a star $S^1_{n-1}$ for some $n \ge 2$; 

\item  $|\mis (T)|=3$ if and only if $T$ is a bistar $S_{p,q}^2$ for some $p,q \geq 1$. 

\item  $|\mis(T)|=4$ if and only if $T$ is a $S_{p,q}^3$ for some $p,q \geq 1$.

\item $|\mis(T)|=5$ if and only if $T$ is a $S_{p,q}^4$ for some $p,q \geq 1$ or a $S_{p,q,r}^3$ for some $p,q,r \geq 1$.
\end{enumerate}
\end{lemma}

\begin{proof} \begin{enumerate}
\item It is trivial.

\item  If $T$ is a star $S^1_{n-1}$ for some $n \ge 2$, then $|\mis(T)|=2$. 
On the other hand, if $|\mis(T)|=2$ but $T$ is not a star, then $T$ contains a path on four vertices. 
This implies that $T$ has at least three maximal independent sets and so $|\mis(T)|>2$, which is a contradiction. 

\item  If $T$ is a bistar {$S_{p,q}^2$ for some $p,q \geq 1$}, then $|\mis(T)|=3$. This can be seen as follows.
By definition,  $T$ is a tree obtained from edge $uv$ by adding $p$ pendent vertices incident to $u$, and $q$ pendent vertices incident to $v$.
Let $I_A=N_T(u)$, 
$I_B=N_T(u) \cup N_T(v) \setminus \{u,v\}$, $I_C=N_T(v)$.
 Then $I_A, I_B, I_C$ are three maximal independent sets  of $T$.
See Figure \ref{indset}.

\begin{figure}[h!] 
\begin{center}
\includegraphics[scale=0.6, trim=0cm 1cm 0cm 0cm]{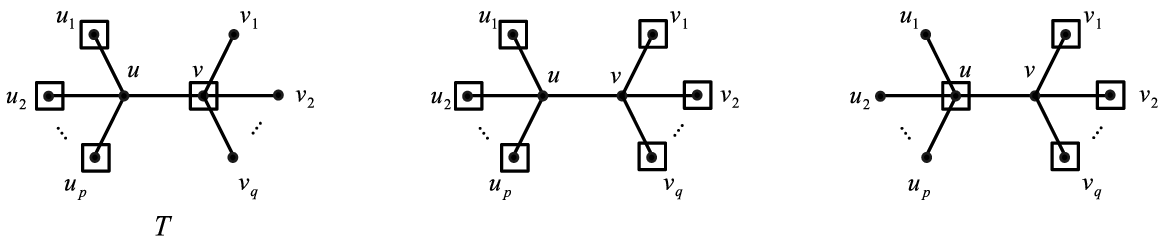}
\end{center}
\caption{\label{indset} {There maximal independent sets $I_A$, $I_B$, and $I_C$ of a bistar $S_{p,q}^2$.}}
\end{figure}
 
On the other hand, if $|\mis(T)|=3$ but $T$ is not a bistar, then $T$ contains a path on five vertices.
This implies that $T$ has at least four maximal independent sets and so so $|\mis(T)|>3$, which is a contradiction. 

\item If $T$ is  a $S_{p,q}^3$ for some $p,q \geq 1$, then $|\mis(T)|=4$.
 This can be seen as follows.
By definition,  $T$ is a tree obtained from a path $v_1v_2v_3$ by adding $p$ pendent vertices incident to $v_1$, 
and $q$ pendent vertices incident to $v_3$.
Let $I_A=N_T(v_1)\cup\{v_3\} \setminus \{v_2\}$, $I_C=N_T(v_3)\cup\{v_1\} \setminus \{v_2\}$, $I_B=N_T(v_1) \cup N_T(v_3)$, and $I_D=\{v_1, v_3\}$. 
Then $I_A, I_B, I_C, I_D$  are four maximal independent sets  of $T$.

On the other hand, if $|\mis(T)|=4$, then $T$ contains a path on five vertices $xv_1v_2v_3y$. Adding extra pendant vertices incident to $v_1$
or extra pendant vertices incident to $v_3$ still results in a tree with four maximal independent sets. 
Adding extra pendant vertices incident to any other vertices will
result in tree with more than four maximal independent sets. 

\item If $T$ is a $S_{p,q}^4$ for some $p,q \geq 1$ or a $S_{p,q,r}^3$ for some $p,q,r \geq 1$, then $|\mis(T)|=5$ which can be seen as follows.

If $S_{p,q}^4$, then by definition,  $T$ is a tree obtained from a path $v_1v_2v_3v_4$ by adding $p$ pendent vertices incident to $v_1$,
and $q$ pendent vertices incident to $v_4$.
Let $I_A=N_T(v_1)\cup N_T(v_4) \setminus \{v_3\}$,  $I_B=N_T(v_1)\cup N_T(v_4)  \setminus \{v_2\}$, 
$I_C=N_T(v_1)\cup\{v_4\}$, $I_D=N_T(v_4)\cup\{v_1\}$, and $I_E=\{v_1, v_4\}$. 
Then $I_A, I_B, I_C, I_D, I_E$  are five maximal independent sets  of $T$.

 If $S_{p,q,r}^3$, then by definition,  $T$ is a tree obtained from a path $v_1v_2v_3$ by adding $p$ pendent vertices incident to $v_1$,
$q$ pendent vertices incident to $v_3$, and $r$ pendent vertices incident to $v_2$.
Let $I_A=N_T(v_1)\cup N_T(v_2) \cup N_T(v_3) \setminus  \{v_1, v_2, v_3\}$,  $I_B=N_T(v_1)\cup N_T(v_2) \setminus  \{v_1, v_2\}$, 
$I_C=N_T(v_1)\cup N_T(v_3)$, $I_D=N_T(v_2)\cup N_T(v_3) \setminus  \{v_2, v_3\}$, and $I_E=N_T(v_2)$. 
Then $I_A, I_B, I_C, I_D, I_E$  are five maximal independent sets  of $T$.

On the other hand, if $|\mis(T)|=5$, then $T$ contains a $S_{p,q}^3$ 
obtained from a path $P =v_1v_2v_3$ on 3 vertices by adding $p$ pendant vertices incident to $v_1$,
and $q$ pendant vertices incident to $v_3$.
There are only two methods to increase the number of maximal independent sets by one.
One method is to subdivide the path $v_1v_2v_3$ by adding an extra vertex,
which will result in $S_{p,q}^4$. 
The other method is  to add extra pendent  vertices incident to $v_2$ which  will result in a $S_{p,q,r}^3$. 
\end{enumerate}
\end{proof} 
 
\smallskip

The \textit{diameter} of a connected graph $G$, denoted as $\diam(G)$, is the length of the longest shortest path between any two vertices of $G$.

\begin{remark}
For a tree $T$, $\normalfont{\diam} (T)=2$ if and only if $T$ is a star, and $\diam (T) = 3$ if and only if $T$ is a bistar.
Hence, for  $k \in \{ 2,3 \}$, the number of maximal independent sets of a tree $T$ is $k$ if and only if $\normalfont{\diam} (T) = k$.
\end{remark}

By Lemma \ref{L:DaisyCube-ResonantGraph-MIS} and Lemma \ref{L:MIS(T)}, 
we can easily determine daisy cubes  that are isomorphic to resonance graphs of plane elementary bipartite graphs and 
have at most five maximal vertices.

\begin{theorem} Let $H$ be a daisy cube that is isomorphic to the resonance graph of a plane elementary bipartite graph with $n \ge 1$ finite faces.
If  $H$ has at most five maximal vertices, then exactly one of the following holds.
\begin{enumerate}
\item  $H$ is isomorphic to $K_2$. 
 
\item  $H$ is isomorphic to an edge disjoint union of $K_2$ and $Q_{n-1}$ with {exactly} one common vertex $0^n$. 

\item  $H$ has three maximal hypercubes $A$, $B$, and $C$  with  {exactly} one common vertex $0^n$ and there are positive integers $p$ and $q$
such that $p+q+2=n$, $\dim(A)=p+1$,  $\dim(B)=p+q$, $\dim(C)=q+1$, where $A$ and $C$ are edge disjoint, 
the intersection of $A$ and $B$ is a $p$-dimensional hypercube, and the intersection of $B$ and $C$ is a $q$-dimensional hypercube. 

\item  $H$ has four maximal hypercubes $A$, $B$, $C$, and $D$  with  {exactly} one common vertex $0^n$ and there are positive integers $p$ and $q$
such that $p+q+3=n$, $\dim(A)=p+1$, $\dim(B)=p+q+1$, $\dim(C)=q+1$,
$\dim(D)=2$, where $A$ and $C$ are edge disjoint, $B$ and $D$ are edge disjoint,
$A \cap B$ is a $p$-dimensional hypercube, $B \cap C$ is a $q$-dimensional hypercube,
$A \cap D$ is an edge, and $C \cap D$ is an edge.

\item $H$ has five maximal hypercubes $A$, $B$, $C$, $D$, and $E$  with  {exactly} one common vertex $0^n$ and there are positive integers $p$ and $q$
such that $p+q+4=n$,  $\dim(A)=\dim(B)=p+q+1$, $\dim(C)=p+2$, $\dim(D)=q+2$,
$\dim(E)=2$, where $C$ and $D$ are edge disjoint, $E$ is an edge disjoint from $A$ and $B$,
$A \cap B$ is a $(p+q)$-dimensional hypercube, $A \cap C$ is a $(p+1)$-dimensional hypercube, $A \cap D$ is a $q$-dimensional hypercube, 
$B \cap C$ is a $p$-dimensional hypercube, $B \cap D$ is a $(q+1)$-dimensional hypercube, 
$C \cap E$ is an edge, and $D \cap E$ is an edge.

\item   $H$ has five maximal hypercubes $A$, $B$, $C$, $D$, and $E$  
with  {exactly} one common vertex $0^n$ and there are positive integers $p$, $q$, and $r$
such that $p+q+r+3=n$, $\dim(A)=p+q+r$, $\dim(B)=p+r+1$, $\dim(C)=p+q+1$, $\dim(D)=q+r+1$,
$\dim(E)=r+2$, where $C$ and $E$ are edge disjoint,
$A \cap B$ is a $(p+r)$-dimensional hypercube, $A \cap C$ is a $(p+q)$-dimensional hypercube, $A \cap D$ is a $(q+r)$-dimensional hypercube, 
$A \cap E$ is a $r$-dimensional hypercube,
$B \cap C$ is a $p$-dimensional hypercube, $B \cap D$ is a $r$-dimensional hypercube, 
$C \cap D$ is a $q$-dimensional hypercube,  $B \cap E$ is an $(r+1)$-dimensional hypercube, 
and $D \cap E$ is a $(r+1)$-dimensional hypercube.
\end{enumerate}
\end{theorem}
\begin{proof} If a daisy cube $H$ is isomorphic to $R(G)$, 
where $G$ is a plane elementary bipartite graph with $n$ finite faces for some positive integer $n$,
then {by Lemma \ref{L:DaisyCube-ResonantGraph-MIS}, $H$ is isomorphic to $D_{\mathcal{I}}(G^*)$, where $G^*$ is a tree.
By Lemma \ref{L:DaisyCubeIndependentSets}, 
we can see that  $I$ is a maximal independent set of $G^*$ if and only if $\chi(I)$ is a maximal vertex of daisy cube $D_{\mathcal{I}}(G^*)$.
Moreover, there is a 1--1 correspondence between the set of vertices of $H$ and the set of independent sets of $G^*$
such that for each vertex $w$ of $H$, there is a  unique independent set $I^*_w$ of $G^*$ satisfying $w=\chi(I^*_w)$.}

By Lemma \ref{L:MIS(T)},  
we distinguish cases based on the number of maximal independent sets of tree $G^*$.
It is trivial when $|\mis(G^*)|=1$.

If  $|\mis(G^*)|=2$, the by Lemma \ref{L:MIS(T)},  $G^*$ is a star $S^1_{n-1}$ for some $n \ge 2$. So,
$G^*$ has two vertex disjoint maximal independent sets: one has a single vertex, the other has $n-1$ vertices.
It follows that $H$ has two edge disjoint maximal hypercubes $K_2$ and $Q_{n-1}$ with  {exactly} one common vertex $0^n$.

If  $|\mis(G^*)|=3$, then  by Lemma \ref{L:MIS(T)}, $G^*$ is a bistar $S_{p,q}^2$  for some $p,q \geq 1$, where $p+q+2=n$, 
see Figure \ref{indset}.  Let $S_{p,q}^2$ be a bistar described in Lemma  \ref{L:MIS(T)}.
Then $S_{p,q}^2$ has three maximal independent sets $I_A=N_T(u) =\{ u_1, \ldots, u_p, v \}$, 
$I_B=N_T(u) \cup N_T(v) \setminus \{u,v\}= \{ u_1,  \ldots, u_p, v_1, \ldots, v_q \}$, $I_C=N_T(v)= \{ u, v_1, \ldots, v_q \}$.

Consequently, $H$ has three maximal hypercubes $A$, $B$, and $C$  
such that $\dim(A) = |I_A|=p+1$,  $\dim(B)=|I_B|=p+q$, and $\dim(C) = |I_C|=q+1$. 
Moreover, $I_A \cap I_C = \emptyset$, $I_A \cap I_B = \{ u_1, u_2, \ldots, u_p \}$, and $I_B \cap I_C = \{ v_1, v_2, \ldots, v_q \}$. 
 Note that hypercubes $A$, $B$, and $C$ have exactly one common vertex $0^n$.
It follows that hypercubes $A$ and $C$ are edge disjoint, 
the intersection of  hypercubes $A$ and $B$ is a $p$-dimensional hypercube $A \cap B$ 
 such that the set of face-labels of its $\Theta$-classes corresponds to $I_A \cap I_C$,
and the intersection of  hypercubes $B$ and $C$ is a $q$-dimensional hypercube $B \cap C$
 such that the set of face-labels of its $\Theta$-classes corresponds to $I_B \cap I_C$. 

If $|\mis(G^*)|=4$ or $5$, it can be proved similarly by Lemma \ref{L:MIS(T)}.
\end{proof}
 
\smallskip

We conclude the section with an open problem.

\bigskip

\noindent
\textbf{Problem.} Let $k > 5$. Characterize all trees $T$ for which $|\mis(T)|=k$.

\bigskip

%The author would like to thank the referees for their helpful comments.

\bmhead{Acknowledgements}

Simon Brezovnik, Niko Tratnik, and Petra \v Zigert Pleter\v sek 
acknowledge the financial support from the Slovenian Research and Innovation Agency: research program 
No.\ P1-0297 (Simon Brezovnik, Niko Tratnik, Petra \v Zigert Pleter\v sek), 
projects No.\ J1-4031 (Simon Brezovnik), N1-0285 (Niko Tratnik), and J7-50226 (Petra \v Zigert Pleter\v sek).  
All four authors thank the Slovenian Research and Innovation Agency for financing our bilateral project 
between Slovenia and the USA (title: \textit{Structural properties of resonance graphs and related concepts}, project No. BI-US/22-24-158).

\bmhead{Competing interests} The authors have no competing interests to declare that are relevant to the content of this article.

\end{document}